\documentclass{amsart}

\usepackage{amssymb}

\usepackage{hyperref}
\usepackage{url}

\newtheorem{theorem}{Theorem}[section]
\newtheorem{lemma}[theorem]{Lemma}
\newtheorem{proposition}[theorem]{Proposition}

\theoremstyle{definition}
\newtheorem{definition}[theorem]{Definition}

\newtheorem{observation}[theorem]{Observation}

\theoremstyle{remark}
\newtheorem{remark}[theorem]{Remark}

\numberwithin{equation}{section}

\DeclareMathOperator{\Aff}{Aff}
\DeclareMathOperator{\Hol}{Hol}
\DeclareMathOperator{\Imaginary}{Im}
\DeclareMathOperator{\Real}{Re}

\DeclareMathOperator{\Cb}{\mathbb{C}}
\DeclareMathOperator{\Nb}{\mathbb{N}}
\DeclareMathOperator{\Pb}{\mathbb{P}}
\DeclareMathOperator{\Rb}{\mathbb{R}}

\DeclareMathOperator{\Cc}{\mathcal{C}}
\DeclareMathOperator{\Nc}{\mathcal{N}}
\DeclareMathOperator{\Oc}{\mathcal{O}}
\DeclareMathOperator{\Uc}{\mathcal{U}}
\DeclareMathOperator{\Vc}{\mathcal{V}}

\newcommand{\abs}[1]{\left|#1\right|}
\newcommand{\norm}[1]{\left\|#1\right\|}
\newcommand{\wh}[1]{\widehat{#1}}

\begin{document}

\title[Gromov hyperbolicity and the Kobayashi metric]{Gromov hyperbolicity, the Kobayashi metric, and $\Cb$-convex sets}

\author{Andrew M. Zimmer}
\address{Department of Mathematics, University of Chicago, Chicago, IL 60637.}
\email{aazimmer@uchicago.edu}
\thanks{This material is based upon work supported by the National Science Foundation under Grant Number NSF 1400919.}


\date{}

\begin{abstract}In this paper we study the global geometry of the Kobayashi metric on domains in complex Euclidean space.  We are particularly interested in developing necessary and sufficient conditions for the Kobayashi metric to be Gromov hyperbolic. For general domains, it has been suggested that a non-trivial complex affine disk in the boundary is an obstruction to Gromov hyperbolicity. This is known to be the case when the set in question is convex. In this paper we first extend this result to $\Cb$-convex sets with $C^1$-smooth boundary. We will then show that some boundary regularity is necessary by producing in any dimension examples of open bounded $\Cb$-convex sets where the Kobayashi metric is Gromov hyperbolic but whose boundary contains a complex affine ball of complex codimension one.
\end{abstract}

\maketitle

\section{Introduction}

In this paper we study the geometry of the Kobasyashi distance $K_{\Omega}$ on domains $\Omega \subset \Cb^d$. Much is known about the behavior of the infinitesimal Kobayashi metric on certain types of domains (see for instance~\cite{JP2013} and the references therein), but very little is known about the global behavior of the Kobayashi distance function. 

It is well known that the unit ball endowed with the Kobayashi metric is isometric to complex hyperbolic space and in particular is an example of a negatively curved Riemannian manifold. One would then suspect that when $\Omega \subset \Cb^d$ is a domain close to the unit ball then $(\Omega, K_\Omega)$ should be negatively curved in some sense. Unfortunately, for general domains the Kobayashi metric is no longer Riemannian and thus will no longer have curvature in a local sense. Instead one can ask if the Kobayashi metric satisfies a coarse notion of negative curvature from geometric group theory called Gromov hyperbolicity.

Gromov hyperbolic metric spaces have been intensively studied and have a number of remarkable properties. For instance:
\begin{enumerate}
\item Iterations of contractions on Gromov hyperbolic metric spaces are very well understood and in particular an analogue of the Wolff--Denjoy theorem always holds~\cite{K2001}, 
\item Given a quasi-isometry $f: X \rightarrow Y$ between two proper geodesic Gromov hyperbolic metric spaces there exists a continuous extension to natural compactifications of $X$ and $Y$ (see for instance~\cite[Chapter III.H, Theorem 3.9]{BH1999}), 
\item (Geodesic shadowing) Every quasi-geodesic is within a bounded distance of an actual geodesic (see for instance~\cite[Chapter III.H, Theorem 1.7]{BH1999}).
\end{enumerate}
In particular, understanding the domains for which the Kobayashi metric is Gromov hyperbolic could lead to new insights about the iteration theory of holomorphic maps and the boundary extension properties of proper holomorphic maps. We should also mention that for the Kobayashi metric it is often easy to construct quasi-geodesics (see for instance Lemma~\ref{lem:quasi_geod} below) but difficult to find actual geodesics. Thus the geodesic shadowing property mentioned above can be a useful tool in understanding the Kobayashi distance function.

Balogh and Bonk~\cite{BB2000} proved that the Kobayashi metric is Gromov hyperbolic when the domain is strongly pseudo-convex. Gaussier and Seshadri~\cite{GS2013} and Nikolov, Thomas, and Trybula~\cite{NTT2014} proved, under certain boundary regularity conditions, that a non-trivial complex affine disk in the boundary of a convex set is an obstruction to the Gromov hyperbolicity of the Kobayashi metric. Extending this work, we recently characterized the bounded convex domains with smooth boundary for which the Kobayashi metric is Gromov hyperbolic:

\begin{theorem}\label{thm:finite_type}\cite{Z2014}
Suppose $\Omega$ is a bounded convex open set with $C^\infty$ boundary. Then $(\Omega, K_{\Omega})$ is Gromov hyperbolic if and only if $\partial \Omega$ has finite type in the sense of D'Angelo.
\end{theorem}

In this paper we explore what happens when convexity is relaxed to $\Cb$-convexity. Recall that an open set $\Omega \subset \Cb^d$ is called \emph{$\Cb$-convex} if its intersection with any complex line is either empty or contractible. See~\cite{APS2004} for the basic properties of these sets. 

Estimates on the infinitesimal Kobayashi metric were established for $\Cb$-convex sets in~\cite{NPZ2011}. In particular, the Bergman, Carath{\'e}odory, and Kobayashi metrics are all bi-Lipschitz~\cite[Proposition 1, Theorem 12]{NPZ2011} for $\Cb$-convex sets which do not contain any complex affine lines. Since Gromov hyperbolicity is an quasi-isometry invariant, the results of this paper could be stated for the Bergman or Carath{\'e}odory metrics instead of the Kobayashi metric. 

It has been suggested by several authors~\cite{BB2000, B2014, GS2013} that ``flatness'' in the boundary is an obstruction to the Kobayashi metric being Gromov hyperbolic. When the domain in question is convex this is indeed the case:

\begin{theorem}\label{thm:no_flats}\cite{GS2013, NTT2014, Z2014}
Suppose $\Omega \subset \Cb^d$ is a bounded convex open set and $(\Omega, K_{\Omega})$ is Gromov hyperbolic. If $\Delta \subset \Cb$ is the unit disk, then every holomorphic map $\varphi: \Delta \rightarrow \partial \Omega$ is constant. 
\end{theorem}

\begin{remark} \ \begin{enumerate}
\item Theorem~\ref{thm:no_flats} was proven when $\Omega$ is bounded and $\partial \Omega$ is $C^\infty$ in~\cite{GS2013}, when $\partial \Omega$ is $C^{1,1}$ and $d=2$ in~\cite{NTT2014}, and in full generality in~\cite{Z2014}. 
\item By~\cite{FS1998}, if $\Omega$ is a convex set then $\partial \Omega$ contains a non-trivial complex affine disk if and only if $\partial \Omega$ contains a non-trivial holomorphic disk. 
\item The above theorem also holds for convex open sets which do no not contain any complex affine lines, see~\cite{Z2014}.
\end{enumerate}
 \end{remark}
 
The first part of this paper is devoted to proving an extension of Theorem~\ref{thm:no_flats} for bounded $\Cb$-convex sets whose boundary is $C^1$.  

\begin{theorem}\label{thm:no_flats_C-convexity}
Suppose $\Omega \subset \Cb^d$ is a bounded $\Cb$-convex open set, $\partial \Omega$ is a $C^1$ hypersurface, and $(\Omega, K_{\Omega})$ is Gromov hyperbolic. If $\Delta \subset \Cb$ is the unit disk, then every holomorphic map $\varphi: \Delta \rightarrow \partial \Omega$ is constant. 
\end{theorem}

\begin{remark} \ \begin{enumerate}
\item In Subsection~\ref{subsec:bd_domain} we will construct examples of bounded $\Cb$-convex sets $\Omega$ where $(\Omega, K_{\Omega})$ is Gromov hyperbolic but $\partial \Omega$ contains a complex affine ball of codimension one, thus showing that some boundary regularity is required. 
\item By~\cite[Proposition 7]{NPZ2011}, if $\Omega$ is a bounded $\Cb$-convex open set with $C^1$ boundary then $\partial \Omega$ contains a non-trivial complex affine disk if and only if $\partial \Omega$ contains a non-trivial holomorphic disk. 
\end{enumerate}
\end{remark} 

In the second part of this paper we will demonstrate some new examples of unbounded convex domains for which the Kobayashi metric is Gromov hyperbolic. By taking projective images of these sets we will obtain bounded $\Cb$-convex sets. Our main motivation for these constructions is Proposition~\ref{prop:c_convex_disks} below, which shows that for any $d > 2$ there is a bounded $\Cb$-convex open set $\Omega \subset \Cb^d$ such that $(\Omega, K_{\Omega})$ is Gromov hyperbolic and $\partial \Omega$ contains a complex affine ball of dimension $d-1$.

Suppose $F:\Cb^d \rightarrow \Rb_{\geq 0}$ is a continuous function which is $C^\infty$ on $\Cb^d \setminus \{0\}$, non-negative, convex, and $F(0)=0$. Define 
\begin{align*}
\Omega_F := \{ (z_0, \dots, z_d) \in \Cb^{d+1} : \Imaginary(z_0) > F(z_1, \dots, z_d) \}
\end{align*}
and 
\begin{align*}
r_F(z_0,\dots, z_d) = F(z_1, \dots, z_d)-\Imaginary(z_0).
\end{align*}
We say that $\Omega_F$ has \emph{finite type away from 0} if for each non-trivial affine line $\ell:\Cb \rightarrow \Cb^{d+1}$ with $\ell(0) \in \partial \Omega_F \setminus \{0\}$ there exists multi-indices $\alpha,\beta$ such that 
\begin{align*}
\frac{\partial^{\abs{\alpha}+\abs{\beta}}}{\partial z^{\alpha} \partial \overline{z}^{\beta}} (r_F \circ \ell)(0) \neq 0.
\end{align*}
If $\delta_1, \dots, \delta_d > 0$, we say that $F$ is \emph{$(\delta_1,\dots, \delta_d)$-homogeneous} if 
\begin{align*}
\frac{1}{t} F( t^{\delta_1} z_1, \dots, t^{\delta_d} z_d) = F(z_1,\dots, z_d)
\end{align*}
for every $t >0$ and $(z_1,\dots, z_d) \in \Cb^d$. This implies that  $\Omega_F$ is invariant under the group 
\begin{align*}
G=\left\{\begin{pmatrix}
t & & & \\
& t^{\delta_1} & & \\
& & \ddots & \\
& & & t^{\delta_d} 
\end{pmatrix} : t > 0\right\}.
\end{align*}

Barth~\cite{B1980} proved that the Kobayashi metric is complete on a convex set $\Omega$ if and only if $\Omega$ does not contain any complex affine lines. Motivated by this fact and language from real projective geometry (see for instance~\cite{B2008}), we say a convex set $\Omega \subset \Cb^d$ is \emph{$\Cb$-proper} if $\Omega$ does not contain any complex affine lines. 

Finally with all this language we will prove the following:

\begin{theorem}\label{thm:main}
Suppose $F:\Cb^d \rightarrow \Rb_{\geq 0}$ is a continuous function which is $C^\infty$ on $\Cb^d \setminus \{0\}$, non-negative, convex, and $F(0)=0$. If $\Omega_F$ is $\Cb$-proper, has finite type away from 0, and $F$ is $(\delta_1,\dots, \delta_d)$-homogeneous, then $(\Omega_F, K_{\Omega_F})$ is Gromov hyperbolic. 
\end{theorem}

We should emphasize that for a general domain it is difficult to determine what geodesics in the Kobayashi metric look like. Thus establishing that the Kobayashi metric is Gromov hyperbolic is a non-trivial task and prior to the work in~\cite{Z2014} was (to the best of our knowledge) only established for bounded strongly pseudo-convex sets~\cite{BB2000} and for certain complex ellipses~\cite{GS2013}. The proof of Theorem~\ref{thm:main} relies heavily on the techniques used in~\cite{Z2014}.

\subsection{Homogeneous polynomial domains} Theorem~\ref{thm:main} applies to convex homogeneous polynomial domains. We say a convex polynomial $P:\Cb^d \rightarrow \Rb_{\geq 0}$ is \emph{non-degenerate} if $P^{-1}(0)$ does not contain any complex affine lines. Theorem~\ref{thm:main} then implies the following:

\begin{theorem}\label{thm:poly}
Suppose $P: \Cb^d \rightarrow \Rb$ is a non-negative, non-degenerate, convex, homogeneous polynomial with $P(0)=0$. If 
\begin{align*}
\Omega = \{ (z_0, z_1, \dots, z_d) \in \Cb^{d+1} : \Imaginary(z_0) > P(z_1, \dots, z_d) \},
\end{align*}
then $(\Omega, K_{\Omega})$ is Gromov hyperbolic. 
\end{theorem}

\subsection{Cones}

Theorem~\ref{thm:main} also applies to many convex cones. Suppose $F:\Cb^d \rightarrow \Rb_{\geq 0}$ is a continuous function which is $C^\infty$ on $\Cb^d \setminus \{0\}$, non-negative, convex, and $F(0)=0$. If
\begin{align*}
\frac{1}{t}F(t z_1, \dots, tz_d) = F(z_1, \dots, z_d) 
\end{align*}
for all $z_1, \dots, z_d \in \Cb$ and $t > 0$, then $\Omega_F$ is a convex cone. And if $\Omega_F$ is a $\Cb$-proper and has finite type away from 0, then $\Omega_F$ satisfies the hypothesis of Theorem~\ref{thm:main} .

To give a concrete example of this construction let $\norm{z}_p$ be the $L_p$-norm on $\Cb^d$ then:

\begin{theorem}\label{thm:cones} If $p > 1$ and
\begin{align*}
\Cc_p = \{ (z_0, z) \in \Cb^{d+1} : \Imaginary(z_0) > \norm{z}_p  \},
\end{align*}
then $(\Cc_p, K_{\Cc_p})$ is Gromov hyperbolic.
\end{theorem}

\subsection{Bounded domains}\label{subsec:bd_domain}

We can also use Theorem~\ref{thm:main} to constructed examples of bounded sets where the Kobayashi metric is Gromov hyperbolic. This can be accomplished by taking the image of one of the domains above under an appropriate projective transformation. Unfortunately convexity is not preserved under such maps and instead the resulting domains are only $\Cb$-convex. 

For instance, consider the cone $\Cc_2$ from Theorem~\ref{thm:cones} and the transformation $f : \Cb^{d+1} \setminus \{ z_0=-i\} \rightarrow \Cb^{d+1}$ given by 
\begin{align*}
f(z_0,\dots,  z_d) = \left( \frac{1}{z_0+i},\frac{z_1}{z_0+i},\dots, \frac{z_d}{z_0+i} \right).
\end{align*}
Notice that $f$ is a restriction of a projective automorphism $\Pb(\Cb^{d+2}) \rightarrow \Pb(\Cb^{d+2})$. In particular $\Omega:=f(\Cc_2)$ is bi-holomorphic to $\Cc_2$ and hence $(\Omega, K_{\Omega})$ is Gromov hyperbolic. Since $\Cc_2$ is convex and $f$ is a projective automorphism the intersection of $\Omega$ with any complex line is either empty or simply connected. Thus $\Omega$ is a  $\Cb$-convex set. Moreover $\Omega$ is bounded and the boundary of $\Omega$ contains $\{0\} \times \{ (z_1, \dots, z_d) : \sum \abs{z_i}^2 < 1\}$. Summarizing the above example:

\begin{proposition}\label{prop:c_convex_disks}
For any $d \geq 2$, there is a bounded $\Cb$-convex open set $\Omega \subset \Cb^d$ such that $(\Omega, K_{\Omega})$ is Gromov hyperbolic and $\partial \Omega$ contains a complex affine ball of dimension $d-1$.
\end{proposition}

\subsection{Corners and the Hilbert metric}
Every proper open convex set $\Omega \subset \Rb^d$ has a projectively invariant metric $H_{\Omega}$ called the \emph{Hilbert metric}. This metric is usually defined using cross ratios, but it has an equivalent formulation which makes it a real projective analogue of the Kobayashi metric (see for instance~\cite{K1977} or~\cite{L1986} or~\cite[Section 3.4]{G2015}). Thus results about the Hilbert metric can serve as guide to understanding the Kobayashi metric. 

The convex domains for which the Hilbert metric is Gromov hyperbolic are very well understood. Karlsson and Noskov showed:

\begin{theorem}\cite{KN2002}\label{thm:hilbert} 
If $\Omega \subset \Rb^d$ is a convex set and $(\Omega, H_{\Omega})$ is Gromov hyperbolic, then $\partial \Omega$ is a $C^1$ hypersurface and $\Omega$ is strictly convex (that is $\partial \Omega$ does not contain any line segments). 
\end{theorem}

\begin{remark}\label{rmk:hilbert} Notice that $\partial\Omega$ being $C^1$ is a equivalent to there being a unique supporting real hyperplane through each boundary point and $\Omega$ being strictly convex is equivalent to each supporting real hyperplane intersecting $\partial \Omega$ at exactly one point. 
\end{remark}

Improving on Theorem~\ref{thm:hilbert}, Benoist~\cite{B2003} characterized the convex domains for which the Hilbert metric is Gromov hyperbolic in terms of the first derivatives of local defining functions for $\partial \Omega$. 

It is natural to ask if some analogue of Theorem~\ref{thm:hilbert} holds for the Kobayashi metric. Since every bounded convex set in $\Cb$ has Gromov hyperbolic Kobayashi metric, in general $\partial \Omega$ need not be $C^1$ and $\Omega$ need not be strictly convex.  However the conclusion of Theorem~\ref{thm:no_flats} can be seen as a complex analytic version of strict convexity. 

Based on Theorem~\ref{thm:hilbert}, Remark~\ref{rmk:hilbert}, and Theorem~\ref{thm:no_flats} it is natural to ask if the number of complex supporting hyperplanes through a point in the boundary is restricted by the Gromov hyperbolicity of the Kobayashi metric. However if $f: \Cb^d \rightarrow \Cb$ is a linear map and $\abs{f(z)} \leq \abs{z}$ for all $z \in \Cb^d$ then the complex hyperplane
\begin{align*}
H_{f} = \{ (f(z), z) : z \in \Cb^{d}\}
\end{align*}
is tangent to the cone $\Cc_2$ (in Theorem~\ref{thm:cones}) at $0$. In particular we see that the set of supporting complex hyperplanes of $\Cc_2$ through $0$ contains a complex ball of dimension $d$.

\subsection*{Acknowledgments} 

I would like to thank the referee for a number of comments and corrections which greatly improved the present work. 

\section{Preliminaries}\label{sec:prelim}

\subsection{Basic notation} We now fix some very basic notations.

\begin{itemize}
\item Let $\Delta: = \{ z \in \Cb : \abs{z} < 1\}$.
\item For $z \in \Cb^d$ let $\norm{z}$ denote the standard Euclidean norm of $z$. 
\item For $z_0 \in \Cb^d$ and $R>0$ let $B_R(z_0) := \{ z \in \Cb^d : \norm{z-z_0} < R\}$.
\item Given a open set $\Omega \subset \Cb^d$ and $p \in \Omega$ let 
\begin{align*}
\delta_{\Omega}(p):= \inf \left\{ \norm{q-p} : q \in  \partial \Omega \right\}.
\end{align*}
\item Given a open set $\Omega \subset \Cb^d$, $p \in \Omega$, and $v \in \Cb^d$ let 
\begin{align*}
\delta_{\Omega}(p;v):= \inf \left\{ \norm{q-p} : q \in (p+\Cb \cdot v)  \cap \partial \Omega \right\}.
\end{align*}
\item Given two open sets $\Omega_1 \subset \Cb^{d_1}$ and $\Omega_2 \subset \Cb^{d_2}$ let $\Hol(\Omega_1, \Omega_2)$ be the space of holomorphic maps from $\Omega_1$ to $\Omega_2$. 
\item Given two open sets $\Oc_1 \subset \Rb^{d_1}$, $\Oc_2 \subset \Rb^{d_2}$, a $C^1$ map $F: \Oc_1 \rightarrow \Oc_2$, and a point $x \in \Oc_1$ define the derivative $d(F)_x : \Rb^{d_1} \rightarrow \Rb^{d_2}$ of $F$ at $x$ by 
\begin{align*}
d(F)_x(v) := \left.\frac{d}{dt}\right|_{t=0} F(x+tv).
\end{align*}
\end{itemize}

\subsection{The Kobayashi metric and distance} Given a domain $\Omega \subset \Cb^d$ the \emph{(infinitesimal) Kobayashi metric} is the pseudo-Finsler metric
\begin{align*}
k_{\Omega}(x;v) = \inf \left\{ \abs{\xi} : f \in \Hol(\Delta, \Omega), \ f(0) = x, \ d(f)_0(\xi) = v \right\}.
\end{align*}
By a result of Royden~\cite[Proposition 3]{R1971} the Kobayashi metric is an upper semicontinuous function on $\Omega \times \Cb^d$. In particular if $\sigma:[a,b] \rightarrow \Omega$ is an absolutely continuous curve (as a map $[a,b] \rightarrow \Cb^d$), then the function 
\begin{align*}
t \in [a,b] \rightarrow k_\Omega(\sigma(t); \sigma^\prime(t))
\end{align*}
is integrable and we can define the \emph{length of $\sigma$} to  be
\begin{align*}
\ell_\Omega(\sigma)= \int_a^b k_\Omega(\sigma(t); \sigma^\prime(t)) dt.
\end{align*}
One can then define the \emph{Kobayashi pseudo-distance} to be
\begin{multline*}
 K_\Omega(x,y) = \inf \left\{\ell_\Omega(\sigma) : \sigma\colon[a,b]
 \rightarrow \Omega \text{ is absolutely continuous}, \right. \\
 \left. \text{ with } \sigma(a)=x, \text{ and } \sigma(b)=y\right\}.
\end{multline*}
This definition is equivalent to the standard definition of $K_\Omega$ via analytic chains, see~\cite[Theorem 3.1]{V1989}.

A nice introduction to the Kobayashi metric and its basic properties can be found in~\cite{K2005}.
 
\subsection{Gromov hyperbolic metric spaces}\label{sec:prelim_gromov}

Suppose $(X,d)$ is a metric space. If $I \subset \Rb$ is an interval, a curve $\sigma: I \rightarrow X$ is a \emph{geodesic} if $d(\sigma(t_1),\sigma(t_2)) = \abs{t_1-t_2}$ for all $t_1, t_2 \in I$.  A \emph{geodesic triangle} in a metric space is a choice of three points in $X$ and geodesic segments  connecting these points. A geodesic triangle is said to be \emph{$\delta$-thin} if any point on any of the sides of the triangle is within distance $\delta$ of the other two sides. 

\begin{definition}
A proper geodesic metric space $(X,d)$ is called \emph{$\delta$-hyperbolic} if every geodesic triangle is $\delta$-thin. If $(X,d)$ is $\delta$-hyperbolic for some $\delta\geq0$ then $(X,d)$ is called \emph{Gromov hyperbolic}.
\end{definition}

The book by Bridson and Haefliger~\cite{BH1999} is one of the standard references for Gromov hyperbolic metric spaces. 

\section{A lower bound for the Kobayashi distance}

In this section we use an estimate for the  infinitesimal Kobayashi metric established by Nikolov, Pflug, and Zwonek ~\cite{NPZ2011} to obtain an estimate on the Kobayashi distance. Using this estimate on the distance we will demonstrate a basic property of the asymptotic geometry of the Kobayashi distance on $\Cb$-convex sets. 

By considering affine maps of the unit disk into a domain, one obtains the following upper bound on the Kobayashi metric:

\begin{observation}
 Suppose $\Omega \subset \Cb^d$ is an open set. Then 
\begin{align*}
k_{\Omega}(p;v) \leq \frac{\norm{v}}{\delta_{\Omega}(p;v)}
\end{align*}
for $p \in \Omega$ and $v \in \Cb^d$ non-zero.
\end{observation}

For $\Cb$-convex sets Nikolov, Pflug, and Zwonek obatined a lower bound on the Kobayashi metric:

\begin{proposition}\label{prop:convex_lower_bd_1}\cite[Proposition 1]{NPZ2011}
Suppose $\Omega \subset \Cb^d$ is an open $\Cb$-convex set. Then
\begin{align*}
\frac{\norm{v}}{4\delta_{\Omega}(p;v)} \leq k_{\Omega}(p;v)
\end{align*}
for $p \in \Omega$ and $v \in \Cb^d$ non-zero.
 \end{proposition}
 
 Using Proposition~\ref{prop:convex_lower_bd_1} we can obtain a lower bound for the Kobayashi distance. 

\begin{lemma}\label{lem:convex_lower_bd_2}
Suppose $\Omega \subset \Cb^d$ is an open $\Cb$-convex set and $p,q \in \Omega$ are distinct. Let $L$ be the complex line containing $p$ and $q$. If $\xi \in L \setminus L \cap \Omega$, then 
\begin{align*}
 \frac{1}{4}\abs{\log \left( \frac{\norm{q-\xi}}{\norm{p-\xi}} \right)} \leq K_{\Omega}(p,q).
 \end{align*}
 \end{lemma}

 \begin{remark} For our purposes the above estimate suffices, but the more precise estimate 
 \begin{align*}
  \frac{1}{4} \log \left(1+  \frac{\norm{p-q}}{\min\{ \delta_\Omega(p; q-p), \delta_\Omega(q; p-q)\} } \right) \leq K_{\Omega}(p,q)
  \end{align*}
  follows from the proof of Proposition 2 part (ii) in~\cite{NT2015}. 
 \end{remark}

\begin{proof}
Since $p,q, \xi$ are all contained in a single affine line the quantity 
\begin{align*}
\abs{\log \left( \frac{\norm{q-\xi}}{\norm{p-\xi}} \right)}
 \end{align*}
 is invariant under affine transformations. In particular we can assume that $\xi=0$, $p=(p_1,0,\dots,0)$, and $q = (q_1,0,\dots, 0)$. Since $\Omega$ is $\Cb$-convex there exists an complex hyperplane $H$ such that $0 \in H$ but $H \cap \Omega = \emptyset$ (see~\cite[Theorem 2.3.9]{APS2004}). Using another affine transformation we may assume in addition that 
 \begin{align*}
 H=\{ (0,z_1, \dots, z_{d-1}) \in \Cb^d : z_1, \dots, z_{d-1} \in \Cb\}.
 \end{align*}
 Now consider the projection $P:\Cb^d \rightarrow \Cb$ onto the first component. Then 
 \begin{align*}
 K_{\Omega}(p,q) \geq K_{P(\Omega)}(p_1,q_1).
 \end{align*}
 By~\cite[Theorem 2.3.6]{APS2004} the linear image of a $\Cb$-convex set is $\Cb$-convex. Hence $P(\Omega) \subset \Cb$ is a $\Cb$-convex open set. Since $P^{-1}(0)=H$ we see that $P(\Omega)$ does not contain zero. Now suppose that $\sigma: [0,1] \rightarrow P(\Omega)$ is a absolutely continuous curve with $\sigma(0) = p_1$ and $\sigma(1)=q_1$. Then since $0 \in \Cb \setminus P(\Omega)$, Proposition~\ref{prop:convex_lower_bd_1} implies that
 \begin{align*}
 \int_0^1 k_{P(\Omega)}(\sigma(t); \sigma^\prime(t)) dt 
 &\geq \frac{1}{4} \int_0^1 \frac{ \abs{\sigma^\prime(t)}}{\abs{\sigma(t)}} dt \geq \frac{1}{4} \int_0^1 \frac{ \abs{\frac{d}{dt}\abs{\sigma(t)}}}{\abs{\sigma(t)}} dt \\
 & \geq \frac{1}{4}\abs{ \int_0^1 \frac{ \frac{d}{dt}\abs{\sigma(t)}}{\abs{\sigma(t)}} dt } \geq\frac{1}{4}\abs{ \log \left( \frac{\abs{q_1}}{\abs{p_1}} \right)} \\
 & =  \frac{1}{4}\abs{\log \left( \frac{\norm{q-\xi}}{\norm{p-\xi}} \right)}.
 \end{align*}
 Since $\sigma$ was an arbitrary absolutely continuous curve joining $p_1$ to $q_1$ the Lemma follows.
 \end{proof}
 
 Using Lemma~\ref{lem:convex_lower_bd_2} we can obtain some information about the asymptotic geometry of the Kobayashi distance. 
 
 \begin{proposition}\label{prop:asym_geom}
Suppose $x, y \in \partial \Omega$ are distinct and $(p_n)_{n \in \Nb}, (q_m)_{m \in \Nb} \subset \Omega$ are sequences such that $p_n \rightarrow x$, $q_m \rightarrow y$, and
\begin{align*}
\liminf_{n,m \rightarrow \infty} K_{\Omega}(p_n, q_m) < \infty.
\end{align*}
If $L$ is the complex line containing $x$ and $y$, then the interior of $\overline{\Omega} \cap L$ in $L$ contains $x$ and $y$. 
\end{proposition}

\begin{proof}
By passing to subsequences we may suppose that there exists $M < \infty$ such that
\begin{align*}
K_{\Omega}(p_n, q_n) < M
\end{align*}
for all $n \in \Nb$. For each $n$, let $L_n$ be the complex affine line containing $p_n$ and $q_n$. Let 
\begin{align*}
\epsilon_n = \min\{\norm{\xi-p_n} : \xi \in  L_n \setminus \Omega \cap L_n\}
\end{align*}
and $\xi_n \in L_n \setminus \Omega \cap L_n$ be a point closest to $p_n$. Then by Lemma~\ref{lem:convex_lower_bd_2}
\begin{align*}
M &\geq \limsup_{n \rightarrow \infty} K_{\Omega}(p_n, q_n) \geq \limsup_{n \rightarrow \infty} \frac{1}{4} \log \frac{ \norm{q_n - \xi_n}}{\norm{p_n - \xi_n}} \\
& \geq   \limsup_{n \rightarrow \infty} \frac{1}{4} \log \frac{ \norm{q_n -p_n}- \epsilon_n}{\epsilon_n} \\
& \geq   \limsup_{n \rightarrow \infty} \frac{1}{4} \log \frac{ \norm{x -y}- \epsilon_n}{\epsilon_n}.
\end{align*}
Since $x \neq y$ there exists an $\epsilon>0$ such that $B_{\epsilon}(p_n) \cap L_n \subset L_n \cap \Omega$ for all $n$ sufficiently large. Which implies that $x$ is in the interior of $\overline{\Omega} \cap L$ in $L$. The same argument applies to $y$. 
\end{proof}

\section{Proof of Theorem~\ref{thm:no_flats_C-convexity}}

We begin with a sketch of the argument: \newline

\textbf{Idea of Proof:} \emph{We assume, for a contradiction, that $(\Omega, K_{\Omega})$ is Gromov hyperbolic and there exists a non-constant holomorphic map $\varphi: \Delta \rightarrow \partial \Omega$. Since $\partial \Omega$ is $C^1$ the existence of a non-constant holomorphic disk in the boundary implies the existence of a non-trivial affine disk~\cite[Proposition 7]{NPZ2011}. We will show that every inward pointing normal line of $\partial \Omega$ can be parameterized to be a quasi-geodesic and if two such quasi-geodesics terminate in the same non-trivial open affine disk in the boundary then they stay within a uniform bounded distance of each other. We then take limits to show that this also holds for two such quasi-geodesics which terminate in the same non-trivial closed affine disk in the boundary. But if the closed disk is maximal this will contradict Proposition~\ref{prop:asym_geom}.} \newline

If $\Omega \subset \Cb^d$ is an open set, $\partial \Omega$ is a $C^1$ hypersurface, and $x \in \partial \Omega$ let $T_x \partial\Omega \subset \Cb^d$ be the real hyperplane tangent to $\partial \Omega$ at $x$ and let $n_x \in \Cb^d$ be the inward pointing unit normal vector at $x$. 

We will repeatedly use the following observation in the proof of Theorem~\ref{thm:no_flats_C-convexity}:

\begin{observation}\label{obs:C1_domains}
Assume $\Omega \subset \Cb^d$ is a bounded open set and $\partial \Omega$ is a $C^1$ hypersurface. Then:
\begin{enumerate}
\item There exists $C,\epsilon_1 > 0$ so that 
\begin{align*}
\delta_\Omega(x+tn_x) \geq Ct 
\end{align*}
for any $x \in \partial \Omega$ and $t \in (0,\epsilon_1)$.
\item For any $R > 0$ there exists $\epsilon_2 > 0$ so that 
\begin{align*}
\delta_\Omega(x+tn_x; v) \geq Rt 
\end{align*}
for any $x \in \partial \Omega$, $v \in T_x\partial \Omega$, and $t \in (0,\epsilon_2)$.
\end{enumerate}
\end{observation}

\subsection{Quasi-geodesics in Gromov hyperbolic metric spaces} 

\begin{definition}
Suppose $(X,d)$ is a metric space, $A \geq 1$, and $B \geq 0$. If $I \subset \Rb$ is an interval, then a map $\sigma: I \rightarrow X$ is a \emph{$(A,B)$-quasi-geodesic} if 
\begin{align*}
\frac{1}{A} \abs{t-s} - B \leq d(\sigma(s), \sigma(t)) \leq A \abs{t-s} + B
\end{align*}
for all $s,t \in I$.
\end{definition}

We will need a basic property of quasi-geodesics in Gromov hyperbolic metric spaces. 

\begin{proposition}\label{prop:gromov_basic}
Suppose $(X,d)$ is a proper geodesic Gromov hyperbolic metric space. For any $A \geq 1$ and $B \geq 0$ there exists $M \geq 0$ such that if $\sigma_1,\sigma_2: \Rb_{\geq 0} \rightarrow X$ are $(A,B)$-quasi-geodesics and 
\begin{align*}
\liminf_{t \rightarrow \infty} d(\sigma_1(t), \sigma_2) < \infty,
\end{align*}
then 
\begin{align*}
\sup_{t \geq 0} \Big( \max\left\{  d(\sigma_1(t), \sigma_2), \ d(\sigma_2(t), \sigma_1)\right\}\Big) \leq M+2d(\sigma_1(0), \sigma_2(0)).
\end{align*}

\end{proposition}

Before starting the proof of the proposition we will make one observation, but first some notation: A \emph{geodesic rectangle} in a metric space $(X,d)$ is a choice of four geodesic segments 
\begin{align*}
 \sigma_i : [a_i,b_i] \rightarrow X \quad i=1,2,3,4
\end{align*}
so that $\sigma_1(b_1) = \sigma_2(a_2)$, $\sigma_2(b_2) = \sigma_3(a_3)$, $\sigma_3(b_3) = \sigma_4(a_4)$, and $\sigma_4(b_4)=\sigma_1(a_1)$. A geodesic rectangle is said to be \emph{$\delta$-thin} if any point on any of the sides of the rectangle is within distance $\delta$ of the other three sides. By connecting a pair of opposite vertices in a geodesic rectangle by a geodesic and considering the resulting two geodesic triangles the following observation is immediate:

\begin{observation}
 Suppose $(X,d)$ is a proper geodesic $\delta$-hyperbolic metric space. Then every geodesic rectangle in $(X,d)$ is $(2\delta)$-thin.
\end{observation}

\begin{proof}[Proof of Proposition~\ref{prop:gromov_basic}]
Assume that $(X,d)$ is $\delta$-hyperbolic for some $\delta \geq 0$.

By~\cite[Chapter III.H, Theorem 1.7]{BH1999} for any $A \geq 1$ and $B \geq 0$ there exists $M_1 >0$ such that for any $(A,B)$-quasi-geodesic $\sigma: \Rb_{\geq 0} \rightarrow X$ there exists a geodesic $\overline{\sigma}:\Rb_{\geq 0} \rightarrow X$ such that $\sigma(0) = \overline{\sigma}(0)$ and
\begin{align*}
\sup_{t \geq 0} \Big(\max\{ d(\sigma(t), \overline{\sigma}), \ d(\overline{\sigma}(t), \sigma) \}\Big) \leq M_1.
\end{align*}
Thus the proposition reduces to the following claim: there exists $M \geq 0$ such that if $\sigma_1, \sigma_2 : \Rb_{\geq 0} \rightarrow X$ are two geodesic rays and 
\begin{align*}
\liminf_{t \rightarrow \infty} K_{\Omega}(\sigma_1(t), \sigma_2) < \infty
\end{align*}
then 
\begin{align*}
\sup_{t >0} d(\sigma_1(t), \sigma_2(t)) \leq M + 2d(\sigma_1(0), \sigma_2(0)).
\end{align*}
Pick a sequence $T_n \rightarrow \infty$ such that 
\begin{align*}
\sup_{n \in \Nb} d(\sigma_1(T_n), \sigma_2)=C < \infty
\end{align*}
for some $C \geq 0$. Let $\gamma_0$ be a geodesic segment joining $\sigma_1(0)$ to $\sigma_2(0)$. Fix $n>0$ sufficiently large and let $\gamma_n$ be a geodesic joining $\sigma_1(T_n)$ to a closest point $\sigma_2(T_n^\prime)$ on $\sigma_2$. 

Now the geodesic segments $\gamma_0$, $\sigma_1|_{[0,T_1]}$, $\gamma_n$, and $\sigma_2|_{[0,T_1^\prime]}$ form a geodesic rectangle which is $(2\delta)$-thin. But for 
\begin{align*}
t \in \left[d(\sigma_1(0), \sigma_2(0)) + 2 \delta, T_n - (C + 2 \delta)\right]
\end{align*}
the point $\sigma_1(t)$ is not $(2\delta)$-close to either $\gamma_0$ or $\gamma_n$, hence there exists $t^\prime$ such that $d(\sigma_1(t), \sigma_2(t^\prime)) \leq 2 \delta$. Now 
\begin{align*}
2 \delta \geq d(\sigma_1(t), \sigma_2(t^\prime)) \geq \abs{t^\prime - t} - d(\sigma_1(0), \sigma_2(0)).
\end{align*}
Hence $\abs{t^\prime - t} \leq 2 \delta + d(\sigma_1(0), \sigma_2(0))$ and so
\begin{align*}
d(\sigma_1(t), \sigma_2(t)) \leq 4\delta + d(\sigma_1(0), \sigma_2(0))
\end{align*}
for $t \in \left[d(\sigma_1(0), \sigma_2(0)) + 2 \delta, T_n - (C + 2 \delta)\right]$. Hence  
\begin{align*}
d(\sigma_1(t), \sigma_2(t)) \leq 4\delta + 2d(\sigma_1(0), \sigma_2(0))
\end{align*}
for $t \leq T_n - (C + 2 \delta)$. Sending $n \rightarrow \infty$ proves the claim and thus the proposition.
\end{proof}

\subsection{Quasi-geodesics in $\Cb$-convex domains} 

For $\Cb$-convex domains with $C^1$ boundary normal lines can be parametrized as quasi-geodesics:

\begin{lemma}\label{lem:quasi_geod}
Suppose $\Omega \subset \Cb^d$ is a bounded $\Cb$-convex open set and $\partial \Omega$ is a $C^1$ hypersurface. Then there exists $A \geq 1$ and $\epsilon >0$ so that for any $x \in \partial \Omega$ the curve 
\begin{align*}
\sigma_{x}&: \Rb_{\geq 0} \rightarrow \Omega \\
\sigma_{x}&(t)=x+e^{-t}\epsilon n_x
\end{align*}
is a $(A,0)$-quasi-geodesic in $(\Omega, K_{\Omega})$.
\end{lemma}

\begin{proof}
Using Observation~\ref{obs:C1_domains} there exists $C, \epsilon >0$ so that
\begin{align*}
\delta_\Omega(x+tn_x) \geq Ct
\end{align*}
for any $x \in \partial \Omega$ and $t \in (0,\epsilon)$.

Now fix $x \in \partial \Omega$. By Lemma~\ref{lem:convex_lower_bd_2}
\begin{align*}
K_{\Omega}(\sigma_{x}(s),\sigma_{x}(t)) \geq  \frac{1}{4}\abs{ \log \frac{\norm{\sigma_{x}(t)-x}}{\norm{\sigma_{x}(s)-x}}} =  \frac{1}{4} \abs{t-s}.
\end{align*}
And if $s \leq t$ then
\begin{align*}
K_{\Omega}(\sigma_{x}(s),\sigma_{x}(t)) 
&\leq \int_{s}^t k_{\Omega}(\sigma_{x}(r); \sigma_{x}^\prime(r)) dr 
\leq \int_s^t \frac{ \norm{\sigma_{x}^\prime(r)} dr}{\delta_{\Omega}(\sigma_{x}(r))} \\
& \leq  \int_s^t \frac{ e^{-r}\epsilon }{Ce^{-r}\epsilon} dr
 \leq \frac{1}{C}\abs{t-s}.
\end{align*}
So $\sigma_x$ is a $(A,0)$-quasi-geodesic where $A = \max \{ 4, 1/C\}$.
\end{proof}

\subsection{Normal lines in domains with $C^1$ boundary}

\begin{proposition}\label{prop:C1}
Assume $\Omega \subset \Cb^d$ is an open set and $\partial \Omega$ is a $C^1$ hypersurface. Suppose that $L$ is a complex affine line so that $L \cap \partial \Omega$ has non-empty interior in $L$. If $\Uc$ is a connected component of the interior of $L \cap \partial \Omega$ and $x,y \in \Uc$, then 
\begin{align*}
\limsup_{t \searrow 0} K_\Omega(x+tn_x, y+tn_y) < \infty.
\end{align*}
\end{proposition}

We begin the proof of Proposition~\ref{prop:C1} with a lemma:

\begin{lemma} \label{lem:same_pt} Assume $\Omega \subset \Cb^d$ is an open set and $\partial \Omega$ is a $C^1$ hypersurface. If $x \in \partial \Omega$ and $v \in \Cb^d$ is a unit vector with $\angle(n_x, v) < \pi/2$, then 
\begin{align*}
\limsup_{t \searrow 0} K_\Omega(x+tn_x, x+tv) < \infty.
\end{align*}
\end{lemma}

\begin{proof}
By hypothesis $v = \lambda n_x + v^\prime$ where $\lambda \in (0,1]$ and $\angle(v^\prime,n_x) =\pi/2$. Then 
 \begin{align*}
K_\Omega(x+tn_x, x+tv) \leq K_\Omega(x+tn_x, x+t\lambda n_x) + K_\Omega(x+t\lambda n_x, x+tv).
\end{align*}
We will bound each term individually. 

By Observation~\ref{obs:C1_domains} there exists $C, \epsilon_1 >0$ so that 
\begin{align*}
\delta_\Omega(x+tn_x) \geq C t \text{ for } t \in (0,\epsilon_1).
\end{align*}
Then for $t \in (0,\epsilon_1)$, define the curve $\gamma_t: [\lambda,1] \rightarrow \Omega$ by 
\begin{align*}
\gamma_t(s) = x+stn_x.
\end{align*}
Then
\begin{align*}
K_\Omega(x+tn_x, x+t\lambda n_x)
& \leq \int_{\lambda}^1 k_\Omega(\gamma_t(s); \gamma_t^\prime(s)) ds 
\leq \int_{\lambda}^1 \frac{\norm{\gamma_t^\prime(s)}}{\delta_\Omega(\gamma_t(s))} ds  \\
& \leq \int_{\lambda}^1 \frac{1}{Cs} ds = -\frac{1}{C}\log(\lambda).
 \end{align*}
 
 Next let $H = \{ w \in \Cb^d : \angle(w,n_x) =\pi/2\}$. Then $x+H= T_x \partial \Omega$ and so by Observation~\ref{obs:C1_domains} there exists $\epsilon_2 > 0$ so that: 
 \begin{align*}
 \delta_\Omega(x+tn_x; w) \geq \left(\frac{1+\norm{v^\prime}}{\lambda}\right)t
 \end{align*}
for all $w \in H$ and $t \in (0, \epsilon_2)$. Then for $t \in (0,\epsilon_2)$, define the curve $\sigma_t: [0,1] \rightarrow \Omega$ by 
\begin{align*}
\sigma_t(s) = s(x + t\lambda n_x) + (1-s)(x+tv) = x+t\lambda n_x + t(1-s)v^\prime
\end{align*}
Now $\sigma_t^\prime(s) = -tv^\prime \in H$ and so 
 \begin{align*}
 \delta_\Omega(\sigma_t(s); \sigma_t^\prime(s)) \geq t.
\end{align*}
Then 
\begin{align*}
K_\Omega(x+t\lambda n_x, x+tv)
& \leq \int_{0}^1 k_\Omega(\sigma_t(s); \sigma_t^\prime(s)) ds 
\leq \int_{0}^1 \frac{\norm{\sigma_t^\prime(s)}}{\delta_\Omega(\sigma_t(s); \sigma_t^\prime(s))} ds  \\
& \leq \int_{0}^1 \frac{t \norm{v^\prime}}{t} ds = \norm{v^\prime}.
 \end{align*}
Combining the two estimates above we obtain the lemma.
 \end{proof}

\begin{proof}[Proof of Proposition~\ref{prop:C1}] Since $\Uc$ is connected it is enough to show: for any $x \in \Uc$ there exists a neighborhood $\Vc$ of $x$ in $\Uc$ so that:
\begin{align*}
\limsup_{t \searrow 0} K_\Omega(x+tn_x, y+tn_y) < \infty \text{ for all } y \in \Vc.
\end{align*}
So fix $x \in \Uc$. By translating $\Omega$ by an affine isometry we can assume that: $x=0$, $n_x = (i,0, \dots, 0)$, $L=  \{ (0,z,0,\dots,0) : z \in \Cb\}$, and $T_x \partial \Omega= \Rb \times \Cb^{d-1}$. Then since $\partial \Omega$ is $C^1$ there exists $\epsilon >0$ and a $C^1$ function $F:  \Rb \times \Cb^{d-1} \rightarrow \Rb$ such that 
\begin{align*}
\{ z \in \Cb^d : \norm{z}_\infty < \epsilon \} & \cap \Omega 
=  \{ (u+iw, z_2, \dots, z_d) : w > F(u,z_2, \dots, z_d) \}.
\end{align*}
Now since $L=  \{ (0,z,0,\dots,0) : z \in \Cb\}$ we can decrease $\epsilon$ so that 
\begin{align*}
F(0,z_2,0,\dots, 0) =0
\end{align*}
 for $\abs{z_2} < \epsilon$. This implies that 
 \begin{align}
\label{eq:containment}
tn_x + \{ (0,z,0,\dots, 0) : \abs{z} < \epsilon\} \subset \Omega
\end{align}
when $t \in (0,\epsilon)$. Now let $\Vc = \{ (0,z,0,\dots, 0) : \abs{z} < \epsilon/2\}$. By decreasing $\epsilon$ we can assume that for any $y \in \Vc$ we have $\angle(n_y, n_x) < \pi/2$. 

Now fix $y \in \Vc$. Then
\begin{align*}
\limsup_{t \searrow 0} K_\Omega(y+tn_x, y+tn_y) < \infty 
\end{align*}
by Lemma~\ref{lem:same_pt} and so it is enough to show that 
\begin{align*}
\limsup_{t \searrow 0} K_\Omega(x+tn_x, y+tn_x) < \infty.
\end{align*}
For $t \in (0,\epsilon)$ define the curve $\gamma_t : [0,1] \rightarrow \Omega$ by 
\begin{align*}
\gamma_t(s) = s(x+tn_x) + (1-s)(y+tn_x) =(1-s)y + tn_x.
\end{align*}
Then the inclusion in~\ref{eq:containment} implies that 
\begin{align*}
\delta_\Omega(\gamma_t(s));\gamma_t^\prime(s)) \geq \epsilon/2.
\end{align*}
Then for $t \in (0,\epsilon)$ we have
\begin{align*}
K_\Omega(x+t n_x, & y+tn_x) \leq  \int_0^1  k_\Omega(\gamma_t(s); \gamma_t^\prime(s)) ds \\
& \leq \int_0^1  \frac{\norm{\gamma_t^\prime(s)}}{\delta_\Omega(\gamma_t(s));\gamma_t^\prime(s)) } ds 
\leq \frac{2 \norm{y}}{\epsilon}.
 \end{align*}
\end{proof}

\subsection{The proof of Theorem~\ref{thm:no_flats_C-convexity}}

Suppose for a contradiction that $\Omega$ is a bounded $\Cb$-convex domain, $\partial \Omega$ is a $C^1$ hypersurface, $(\Omega, K_{\Omega})$ is Gromov hyperbolic, and there exists a non-trivial holomorphic map $\varphi: \Delta \rightarrow \partial \Omega$. 

By~\cite[Proposition 7]{NPZ2011} there exists an non-trivial affine map $\ell: \Delta \rightarrow \partial \Omega$. Let $L$ be the complex affine line containing $\ell(\Delta)$ and let $\Uc$ be a connected component of the interior of $L \cap \partial \Omega$ in $L$.  

By Lemma~\ref{lem:quasi_geod} there exists $\epsilon >0$ and $A \geq 1$ such that for each $x \in \partial \Omega$  the curve $\sigma_{x}(t) = x+e^{-t} \epsilon n_x$ is an $(A,0)$-quasi-geodesic in $(\Omega, K_{\Omega})$. Let $M$ be as in the statement of Proposition~\ref{prop:gromov_basic} for $(A,0)$-quasi-geodesics. 

Since $\partial \Omega$ is compact, there exists an $R>0$ such that 
\begin{align*}
K_{\Omega}(\sigma_{x}(0), \sigma_{y}(0)) \leq R
\end{align*}
for all $x,y \in \partial \Omega$.

Now if $x,y \in \Uc$ then by Proposition~\ref{prop:C1} we have 
\begin{align*}
\limsup_{t \rightarrow \infty} K_\Omega( \sigma_x(t), \sigma_y(t)) < \infty.
\end{align*}
So by Proposition~\ref{prop:gromov_basic}
\begin{align}
\label{eq:bd_dist}
\sup_{t \geq0} \Big(\max\left\{ K_{\Omega}(\sigma_{y}(t), \sigma_{x}),  K_{\Omega}(\sigma_{x}(t), \sigma_{y})\right\}\Big)  \leq M+2R.
\end{align}
But then the inequality in~\ref{eq:bd_dist} holds for all $x,y \in \overline{\Uc}$. But this contradicts Proposition~\ref{prop:asym_geom} when $x \in \Uc$ and $y \in \partial \Uc$.

\section{$L$-convexity and limits in the local Hausdorff topology} 

When $L \geq 1$, Mercer~\cite{M1993} calls a convex set $\Omega$ \emph{$L$-convex} if there exists $C > 0$ such that
\begin{align*}
\delta_{\Omega}(p;v) \leq C\delta_{\Omega}(p)^{1/L}
\end{align*}
for all $p \in \Omega$ and $v \in \Cb^d$ non-zero. Every bounded strongly convex set is $2$-convex and Mercer~\cite[Section 3]{M1993} extended results about limits of complex geodesics in strongly convex sets (see~\cite[Section 2]{CHL1988}) to general $L$-convex sets. In~\cite{Z2014}, we extended some of these results to sequences of geodesic lines $\sigma_n : \Rb\rightarrow \Omega_n$ when $\Omega_n$ is a sequence of convex sets which converge in the local Hausdorff topology and satisfy a uniform $L$-convex property. In this section we recall some of these results. 

Given a set $A \subset \Cb^d$, let $\Nc_{\epsilon}(A)$ denote the \emph{$\epsilon$-neighborhood of $A$} with respect to the Euclidean distance. The \emph{Hausdorff distance} between two bounded sets $A,B \subset \Cb^d$ is given by
\begin{align*}
d_{H}(A,B) = \inf \left\{ \epsilon >0 : A \subset \Nc_{\epsilon}(B) \text{ and } B \subset \Nc_{\epsilon}(A) \right\}.
\end{align*}
Equivalently, 
\begin{align*}
d_H(A,B) = \max\left\{\sup_{a \in A}\inf_{b \in B} \norm{a-b}, \sup_{b \in B} \inf_{a \in A} \norm{a-b} \right\}.
\end{align*}
The Hausdorff distance induces a topology on the space of bounded open convex sets in $\Cb^d$.

The space of all convex sets in $\Cb^d$ can be given a topology from the local Hausdorff semi-norms. For $R >0$ and a set $A \subset \Cb^d$ let $A^{(R)} := A \cap B_R(0)$. Then define the \emph{local Hausdorff semi-norms} by
\begin{align*}
d_H^{(R)}(A,B) := d_H(A^{(R)}, B^{(R)}).
\end{align*}
A sequence of open convex sets $A_n$ is said to converge in the local Hausdorff topology to an open convex set $A$ if there exists some $R_0 \geq 0$ so that 
\begin{align*}
\lim_{n \rightarrow \infty} d_H^{(R)}(A_n,A) = 0
\end{align*}
for all $R \geq R_0$.

Unsurprisingly, the Kobayashi distance is continuous with respect to the local Hausdorff topology.

\begin{theorem}\cite[Theorem 4.1]{Z2014}\label{thm:Kob_cont}
\label{thm:dist_conv}
Suppose $\Omega_n$ is a sequence of $\Cb$-proper convex open sets converging to a $\Cb$-proper convex open set $\Omega$ in the local Hausdorff topology. Then 
\begin{align*}
K_{\Omega}(x,y) = \lim_{n \rightarrow \infty} K_{\Omega_n}(x,y)
\end{align*}
for all $x,y \in \Omega$ uniformly on compact sets of $\Omega \times \Omega$.
\end{theorem}

\begin{remark} Greene and Krantz~\cite{GK1982} also study the continuity of intrinsic metrics under various notions of convergence of domains.
\end{remark}

In the proof of Theorem~\ref{thm:main} we will be interested in the limit of a sequence of geodesics $\sigma_n : \Rb \rightarrow \Omega_n$ when the target domains converges in the local Hausdorff topology. As the next two Propositions show when the sequence $\Omega_n$ has uniform convexity properties these limits have nice properties. 

\begin{proposition}\cite[Proposition 7.8]{Z2014}\label{prop:m_convex}
Suppose $\Omega_n$ is a sequence of $\Cb$-proper convex open sets converging to a $\Cb$-proper convex open set $\Omega$ in the local Hausdorff topology and there exists $C>0$, $L \geq 1$, and a compact set $K$ satisfying  
\begin{align*}
\delta_{\Omega_n}(p;v) \leq C \delta_{\Omega_n}(p)^{1/L}
\end{align*}
for every $n$ sufficiently large, $p \in \Omega_n \cap K$, and $v \in \Cb^d$ non-zero. 

If $\sigma_n : \Rb \rightarrow \Omega_n$ is a sequence of geodesics and there exists $a_n \leq b_n$ such that 
\begin{enumerate}
\item $\sigma_n([a_n,b_n]) \subset K$,
\item $\lim_{n \rightarrow \infty} \norm{\sigma_n(a_n)-\sigma_n(b_n)} > 0$, 
\end{enumerate}
then there exists $T_n \in [a_n,b_n]$ such that a subsequence of $t\rightarrow \sigma_n(t+T_n)$ converges locally uniformly to a geodesic $\sigma:\Rb \rightarrow \Omega$.  
\end{proposition}

\begin{remark} Proposition 7.8 in~\cite{Z2014} assumes that $K=B_R(0)$ for some $R>0$. The proof taken verbatim implies the more general case stated here.
\end{remark}

\begin{proposition}\cite[Proposition 7.9]{Z2014}\label{prop:limits_exist}
Suppose $\Omega_n$ is a sequence of $\Cb$-proper convex open sets converging to a $\Cb$-proper convex open set $\Omega$ in the local Hausdorff topology and for any $R>r > 0$ there exists $C=C(R,r) >0$ and $L=L(R,r) \geq 1$ such that 
\begin{align*}
\delta_{\Omega_n}(p; v) \leq C\delta_{\Omega_n}(p)^{1/L}
\end{align*}
for all $n$ sufficiently large, $p \in \Omega_n \cap \{ r \leq \norm{z} \leq R\}$, and $v \in \Cb^d$ non-zero. 

Assume $\sigma_n : \Rb \rightarrow \Omega_n$ is a sequence of geodesics converging locally uniformly to a geodesic $\sigma:\Rb \rightarrow \Omega$. If $t_n \rightarrow \infty$ is a sequence such that $\lim_{n \rightarrow \infty} \sigma_n(t_n) = x_{\infty} \in \overline{\Cb^d}$,  then
\begin{align*}
\lim_{t \rightarrow \infty} \sigma(t) = x_{\infty}.
\end{align*}
\end{proposition}

\begin{remark} Proposition 7.9 in~\cite{Z2014} is slightly less general ($r$ is assumed to be zero), but the proof taken essentially verbatim implies the more general case stated here.
\end{remark}

\begin{proposition}\label{prop:finite_type}
Suppose that $F_n:\Cb^d \rightarrow \Rb_{\geq 0}$ is a sequence of continuous functions which are $C^\infty$ on $\Cb^d \setminus \{0\}$, non-negative, convex, and $F_n(0)=0$.  Assume that the $F_n$ converges locally uniformly to a function $F:\Cb^d \rightarrow \Rb_{\geq 0}$ in the $C^0$ topology on $\Cb^d$ and in $C^\infty$ topology on $\Cb^d \setminus \{0\}$. 

If $\Omega_F$ is $\Cb$-proper and has finite type away from 0, then for any $R >  r > 0$ there exists $N=N(R,r) \geq 1$, $L=L(R,r) \geq 1$, and $C=C(R,r) > 0$ such that 
\begin{align*}
\delta_{\Omega_{F_n}}(p;v) \leq C\delta_{\Omega_{F_n}}(p)^{1/L}
\end{align*}
for all $n \geq N$, $p \in \Omega_{F_n} \cap \{ r \leq \norm{z} \leq R\}$ and $v \in \Cb^{d+1}$ non-zero.
\end{proposition}

This is a special case of Proposition 9.3 and Example 9.4 in~\cite{Z2014}.

\section{Proof of Theorem~\ref{thm:main}}

Suppose $F:\Cb^d \rightarrow \Rb_{\geq 0}$ is a continuous function which is $C^\infty$ on $\Cb^d \setminus \{0\}$, non-negative, convex, and $F(0)=0$. As in the introduction let
\begin{align*}
\Omega_F =  \{ (z_0, \dots, z_d) \in \Cb^{d+1} : \Imaginary(z_0) > F(z_1, \dots, z_d) \}.
\end{align*}
Assume that $\Omega_F$ is $\Cb$-proper, has finite type away from 0, and $F$ is $(\delta_1,\dots, \delta_d)$-homogeneous

Suppose for a contradiction that  $(\Omega_F, K_{\Omega_F})$ is not Gromov hyperbolic. Then there exists points $x_n,y_n, z_n \in \Omega_F$, geodesic segments $\sigma_{x_ny_n}, \sigma_{y_nz_n}, \sigma_{z_nx_n}$ joining them, and a point $u_n$ in the image of $\sigma_{x_ny_n}$ such that
\begin{align*}
K_{\Omega_F}(u_n, \sigma_{y_nz_n} \cup \sigma_{z_nx_n}) > n.
\end{align*}
For $t > 0$, $\Omega_F$ is invariant under the holomorphic transformation 
\begin{align*}
(z_0, z_1, \dots, z_d) \rightarrow \left( t z_0, t^{\delta_1} z_1, \dots, t^{\delta_d} z_d\right).
\end{align*}
So we may assume that $\norm{u_n}=1$. By passing to a subsequence we can also assume that $u_n \rightarrow u_\infty \in \overline{\Omega}_F$. Now there are two cases to consider $u_{\infty} \in \partial \Omega_F$ and $u_{\infty} \in \Omega_F$. 

\subsection{Case 1:} Suppose that $u_{\infty} \in \partial \Omega_F$. Since $\norm{u_{\infty}}=1$ we see that $\partial \Omega_F$ has finite line type at $u_\infty$. 

We will need two results from~\cite{Z2014}:

\begin{theorem}\label{thm:gaussier}\cite[Theorem 10.1]{Z2014} 
Suppose $\Omega \subset \Cb^{d+1}$ is a convex open set such that $\partial \Omega$ is $C^L$ and has finite line type $L$ near some $\xi \in \partial \Omega$. If $q_n \in \Omega$ is a sequence converging to $\xi$, then there exists $n_k \rightarrow \infty$ and affine maps $A_k \in \Aff(\Cb^d)$ such that
\begin{enumerate}
\item $A_k \Omega$ converges in the local Hausdorff topology to a $\Cb$-proper convex open set $\wh{\Omega}$ of the form:
\begin{align*}
\wh{\Omega} = \{ (z_0,z_1 \dots, z_d) \in \Cb^{d} : \Real(z_0)  > P(z_1,z_2, \dots, z_d) \}
\end{align*}
where $P$ is a non-negative non-degenerate convex polynomial with $P(0)=0$,
\item $A_k u_{n_k} \rightarrow u_{\infty} \in \wh{\Omega}$, and
\item for any $R>0$ there exists $C=C(R)>0$ and $N=N(R)>0$ such that 
\begin{align*}
\delta_{\Omega_n}(p; v) \leq C \delta_{\Omega_n}(p)^{1/L}
\end{align*}
for all $n>N$, $p \in B_R(0) \cap \Omega_n$, and $v \in \Cb^d$ non-zero. 
\end{enumerate}
\end{theorem}

\begin{remark}
With the exception of part (3) the above theorem follows from an argument of Gaussier~\cite{G1997}. 
\end{remark}

\begin{proposition}\label{prop:geodesics_well_behaved}\cite[Proposition 12.2]{Z2014}
Suppose $\Omega$ is a domain of the form 
\begin{align*}
\Omega = \{ (z_0, \dots, z_d) \in \Cb^{d+1} : \Real(z_0)  > P(z_1, \dots, z_d) \}
\end{align*}
where $P$ is a non-negative non-degenerate convex polynomial with $P(0)=0$. If $\sigma: \Rb \rightarrow \Omega$ is a geodesic, then $\lim_{t \rightarrow -\infty} \sigma(t)$ and $\lim_{t \rightarrow \infty} \sigma(t)$ both exist in $\overline{\Cb^d}$ and are distinct.
\end{proposition}

We can now complete the proof of Theorem~\ref{thm:main} in case 1. By passing to a subsequence there exists affine maps $A_n \in \Aff(\Cb^d)$ and a $\Cb$-proper convex open set $\wh{\Omega}$ such that 

\begin{enumerate}
\item $\Omega_n: = A_n \Omega_F \rightarrow \wh{\Omega}$ in the local Hausdorff topology,
\item $A_n u_n \rightarrow u_{\infty} \in \wh{\Omega}$,
\item for any $R>0$ there exists $C, N, L>0$ such that 
\begin{align*}
\delta_{\Omega_n}(p; v) \leq C \delta_{\Omega_n}(p)^{1/L}
\end{align*}
for all $n>N$, $p \in B_R(0) \cap \Omega_n$, and $v \in \Cb^d$ non-zero,
\item if $\sigma: \Rb \rightarrow \wh{\Omega}$ is a geodesic then $\lim_{t \rightarrow -\infty} \sigma(t)$ and $\lim_{t \rightarrow \infty} \sigma(t)$ both exist in $\overline{\Cb^d}$ and are distinct.
\end{enumerate}
By passing to another subsequence we can suppose that $A_n x_n \rightarrow x_{\infty}$, $A_n y_n \rightarrow y_{\infty}$, and $A_n z_n \rightarrow z_{\infty}$ for some $x_{\infty}, y_{\infty}, z_{\infty} \in \overline{\Cb^d} = \Cb^d \cup \{\infty\}$.

Parametrize $\sigma_{x_ny_n}$ such that $\sigma_{x_ny_n}(0)=u_n$ then using Theorem~\ref{thm:Kob_cont} we can pass to a subsequence such that $A_n\sigma_{x_ny_n}$ converges locally uniformly to a geodesic $\sigma: \Rb \rightarrow \wh{\Omega}$. Moreover, by Proposition~\ref{prop:limits_exist}
\begin{align*}
\lim_{t \rightarrow -\infty} \sigma(t) = x_{\infty} \text{ and } \lim_{t \rightarrow +\infty} \sigma(t) = y_{\infty}.
\end{align*}
Then we must have that $x_{\infty} \neq y_{\infty}$. 

So $z_{\infty}$ does not equal at least one of $x_{\infty}$ or $y_{\infty}$. By relabeling we can suppose that $x_{\infty} \neq z_{\infty}$. Since $x_{\infty} \neq z_{\infty}$ at least one is finite and hence by Proposition~\ref{prop:m_convex} we may pass to a subsequence and parametrize $A_n\sigma_{x_nz_n}$ so that it converges locally uniformly to a geodesic $\wh{\sigma}:\Rb \rightarrow \wh{\Omega}$. But then
\begin{align*}
K_{\wh{\Omega}}(u_{\infty},\wh{\sigma}(0)) 
&= \lim_{n \rightarrow \infty} K_{\Omega_n}(A_nu_n, A_n\sigma_{x_n y_n} (0))\\
&= \lim_{n \rightarrow \infty} K_{\Omega_F}(u_n, \sigma_{x_n y_n} (0))\\
& \geq  \lim_{n \rightarrow \infty} K_{\Omega_F}(u_n, \sigma_{x_nz_n})=\infty
\end{align*}
which is a contradiction. 

\subsection{Case 2:} Suppose that $u_{\infty} \in \Omega_F$. By passing to a subsequence we can suppose that $x_n \rightarrow x_{\infty}$, $y_n \rightarrow y_{\infty}$, and $z_n \rightarrow z_{\infty}$ for some $x_{\infty}, y_{\infty}, z_{\infty} \in \overline{\Cb^{d+1}} = \Cb^{d+1} \cup \{\infty\}$. Since 
\begin{align*}
K_{\Omega_F}(u_n, \{x_n,y_n,z_n\}) > n
\end{align*}
we must have that $x_{\infty}, y_{\infty}, z_{\infty} \in \partial \Omega_F \cup \{\infty\}$. 

Fix $R > r > 0$ then by Proposition~\ref{prop:finite_type} there exists $C >0$ and $L \geq 1$ such that 
\begin{align*}
\delta_{\Omega_F}(p; v) \leq C \delta_{\Omega_F}(p)^{1/L}
\end{align*}
for all $p \in \Omega_F \cap \{ r \leq \norm{z} \leq R\}$ and $v \in \Cb^{d+1}$ non-zero. In particular we can apply Proposition~\ref{prop:m_convex} and Proposition~\ref{prop:limits_exist} when taking the limits of geodesics $\sigma_n : \Rb \rightarrow \Omega_F$. 

Now parametrize $\sigma_{x_ny_n}$ such that $\sigma_{x_ny_n}(0)=u_n$ then using the Arzel{\`a}-Ascoli theorem we can pass to a subsequence such that $\sigma_{x_ny_n}$ converges locally uniformly to a geodesic $\sigma: \Rb \rightarrow \Omega_F$. Then by Proposition~\ref{prop:limits_exist}
\begin{align*}
\lim_{t \rightarrow -\infty} \sigma(t) = x_{\infty} \text{ and } \lim_{t \rightarrow +\infty} \sigma(t) = y_{\infty}.
\end{align*}

We now claim:

\begin{proposition}\label{prop:well_behaved}
If $\gamma: \Rb \rightarrow \Omega_F$ is a geodesic, then $\lim_{t \rightarrow -\infty} \gamma(t)$ and $\lim_{t \rightarrow \infty} \gamma(t)$ both exist in $\overline{\Cb^{d+1}}$ and are distinct.
\end{proposition}

Delaying the proof of Proposition~\ref{prop:well_behaved} we can complete the proof of Theorem~\ref{thm:main}. Now $x_{\infty} \neq y_{\infty}$, so $z_{\infty}$ does not equal at least one of $x_{\infty}$ or $y_{\infty}$. By relabeling we can suppose that $x_{\infty} \neq z_{\infty}$. Since $x_{\infty} \neq z_{\infty}$ at least one is finite and hence by Proposition~\ref{prop:m_convex} we may pass to a subsequence and parametrize $\sigma_{x_nz_n}$ so that it converges locally uniformly to a geodesic $\wh{\sigma}:\Rb \rightarrow \Omega_F$. But then
\begin{align*}
K_{\Omega_F}(u_{\infty},\wh{\sigma}(0)) 
&= \lim_{n \rightarrow \infty} K_{\Omega_F}(u_n, \sigma_{x_n y_n} (0))\\
& \geq  \lim_{n \rightarrow \infty} K_{\Omega_F}(u_n, \sigma_{x_nz_n})=\infty
\end{align*}
which is a contradiction. 

To finish the argument we now prove Proposition~\ref{prop:well_behaved} which will require one result from~\cite{Z2014}. Define the \emph{Gromov product} on $\Omega$ by
\begin{align*}
(p|q)_o^{\Omega} = \frac{1}{2} \Big(K_{\Omega}(p,o)+K_{\Omega}(o,q)-K_{\Omega}(p,q) \Big).
\end{align*}

\begin{proposition}\cite[Proposition 11.3]{Z2014}\label{prop:gromov_prod_finite}
Suppose $\Omega$ is a $\Cb$-proper convex open set. Assume $p_n, q_n \in \Omega$ are sequences with $\lim_{n \rightarrow \infty} p_n = \xi^+ \in \partial \Omega$, $\lim_{n \rightarrow \infty} q_n = \xi^- \in \partial \Omega \cup \{ \infty\}$, and 
\begin{align*}
\liminf_{n,m \rightarrow \infty} \ (p_n | q_m)_o^{\Omega} < \infty.
\end{align*}
If $\partial \Omega$ is $C^2$ near $\xi^+$, then $\xi^+ \neq \xi^-$.
\end{proposition} 

\begin{proof}[Proof of Proposition~\ref{prop:well_behaved}]
First suppose that $\lim_{t \rightarrow \infty} \gamma(t)$ does not exist. Then there exists $s_n,t_n \rightarrow \infty$ and $\xi_1, \xi_2 \in \Cb^{d+1} \cup \{\infty\}$ such that 
\begin{align*}
\lim_{n \rightarrow \infty} \gamma(s_n) = \xi_1 \neq \xi_2 = \lim_{n \rightarrow \infty} \gamma(t_n).
\end{align*}
Now  (up to relabeling) either $\xi_1 =0$ and $\xi_2 = \infty$ or $\xi_1 \in \partial \Omega_F \setminus \{0\}$. In either case there exists $R > r > 0$ and $[a_n, b_n] \subset [ \min\{s_n,t_n\}, \infty)$ such that  
\begin{enumerate}
\item $\gamma([a_n,b_n]) \subset \{ r < \norm{z} < R\}$ for large $n$,
\item $\lim_{n \rightarrow \infty} \norm{\gamma(a_n)-\gamma(b_n)} > 0$. 
\end{enumerate}
So by Proposition~\ref{prop:m_convex} we can pass to a subsequence and find $T_n \in [a_n, b_n]$ such that $t \rightarrow \sigma(t+T_n)$ converges locally uniformly to a geodesic $\wh{\gamma}:\Rb \rightarrow \Omega_F$. But then 
\begin{align*}
K_{\Omega_F}(\gamma(0), \wh{\gamma}(0)) = \lim_{n \rightarrow\infty} K_{\Omega_F}(\gamma(0), \gamma(T_n)) \geq \limsup_{n \rightarrow \infty} a_n \geq
 \lim_{n \rightarrow \infty} \min\{t_n, s_n\} = \infty
\end{align*}
which is a contradiction. Thus the limits $\lim_{t \rightarrow -\infty} \gamma(t)$ and $\lim_{t \rightarrow \infty} \gamma(t)$ both exist.

Now suppose for a contradiction that $\lim_{t \rightarrow -\infty} \gamma(t) = \xi = \lim_{t \rightarrow \infty} \gamma(t)$. Notice that
\begin{align*}
\lim_{t \rightarrow \infty} (\gamma(t)|\gamma(-t))_{\gamma(0)}^{\Omega} = \lim_{t \rightarrow \infty} 0 = 0,
\end{align*}
then since $\partial \Omega_F \setminus \{0\}$ is $C^\infty$ Proposition~\ref{prop:gromov_prod_finite} implies that either $\xi=0$ or $\xi=\infty$. 

If $\xi=0$ consider the geodesics 
\begin{align*}
\gamma_n(t):= 
\begin{pmatrix} n & & & \\
& n^{\delta_1} & & \\
& & \ddots & \\
& & & n^{\delta_d} 
\end{pmatrix}\gamma(t).
\end{align*}
Notice that 
\begin{align*}
\lim_{t \rightarrow -\infty} \gamma_n(t) = 0 = \lim_{t \rightarrow \infty} \gamma_n(t)
\end{align*}
for any $n$, but $\lim_{n \rightarrow \infty} \gamma_n(0) = \infty$. Thus by Proposition~\ref{prop:m_convex} we can pass to a subsequence $n_k \rightarrow \infty$ and find $\alpha_k \in (-\infty, 0]$ and $\beta_k \in [0,\infty)$ such that the geodesics $t \rightarrow \gamma_{n_k}(t+\alpha_k)$ and $t \rightarrow \gamma_{n_k}(t+\beta_k)$ converge locally uniformly to geodesics $\wh{\gamma}_1, \wh{\gamma}_2 : \Rb \rightarrow \Omega_F$. Since $\gamma_n(0) \rightarrow \infty$ and $0 \in \partial \Omega_F$ Lemma~\ref{lem:convex_lower_bd_2} implies that $\alpha_k \rightarrow -\infty$ and $\beta_k \rightarrow \infty$. Then 
\begin{align*}
K_{\Omega_F}(\wh{\gamma}_1(0), \wh{\gamma}_2(0)) = \lim_{k \rightarrow \infty} K_{\Omega_F}(\gamma_{n_k}(\alpha_k), \gamma_{n_k}(\beta_k)) = \lim_{k \rightarrow \infty} \beta_k - \alpha_k = \infty
\end{align*}
which is a contradiction. 

The case in which $\xi=\infty$ is shown to be impossible in a similar fashion by considering the geodesics \begin{align*}
\gamma_n(t):= 
\begin{pmatrix} n^{-1} & & & \\
& n^{-\delta_1} & & \\
& & \ddots & \\
& & & n^{-\delta_d} 
\end{pmatrix}\gamma(t).
\end{align*}
Thus $\lim_{t \rightarrow -\infty} \gamma(t)$ and $\lim_{t \rightarrow \infty} \gamma(t)$ both exist in $\overline{\Cb^{d+1}}$ and are distinct.
\end{proof}

\bibliographystyle{amsplain}
\bibliography{complex_kob} 

\providecommand{\bysame}{\leavevmode\hbox to3em{\hrulefill}\thinspace}
\providecommand{\MR}{\relax\ifhmode\unskip\space\fi MR }
\providecommand{\MRhref}[2]{%
  \href{http://www.ams.org/mathscinet-getitem?mr=#1}{#2}
}
\providecommand{\href}[2]{#2}
\begin{thebibliography}{10}

\bibitem{APS2004}
Mats Andersson, Mikael Passare, and Ragnar Sigurdsson, \emph{Complex convexity
  and analytic functionals}, Progress in Mathematics, vol. 225, Birkh\"auser
  Verlag, Basel, 2004. \MR{2060426 (2005a:32011)}

\bibitem{BB2000}
Zolt{\'a}n~M. Balogh and Mario Bonk, \emph{Gromov hyperbolicity and the
  {K}obayashi metric on strictly pseudoconvex domains}, Comment. Math. Helv.
  \textbf{75} (2000), no.~3, 504--533. \MR{1793800 (2001k:32046)}

\bibitem{B1980}
Theodore~J. Barth, \emph{Convex domains and {K}obayashi hyperbolicity}, Proc.
  Amer. Math. Soc. \textbf{79} (1980), no.~4, 556--558. \MR{572300 (81g:32018)}

\bibitem{B2003}
Yves Benoist, \emph{Convexes hyperboliques et fonctions quasisym\'etriques},
  Publ. Math. Inst. Hautes \'Etudes Sci. (2003), no.~97, 181--237. \MR{2010741
  (2005g:53066)}

\bibitem{B2008}
\bysame, \emph{A survey on divisible convex sets}, Geometry, analysis and
  topology of discrete groups, Adv. Lect. Math. (ALM), vol.~6, Int. Press,
  Somerville, MA, 2008, pp.~1--18. \MR{2464391 (2010h:52013)}

\bibitem{BH1999}
Martin~R. Bridson and Andr{\'e} Haefliger, \emph{Metric spaces of non-positive
  curvature}, Grundlehren der Mathematischen Wissenschaften [Fundamental
  Principles of Mathematical Sciences], vol. 319, Springer-Verlag, Berlin,
  1999. \MR{1744486 (2000k:53038)}

\bibitem{B2014}
S.~{Buckley}, \emph{Gromov hyperbolicity of invariant metrics},
  \url{http://www.uma.es/investigadores/grupos/cfunspot/research/0806pBuckley.pdf},
  2008, Accessed: 2016-01-12.

\bibitem{CHL1988}
Chin-Huei Chang, M.~C. Hu, and Hsuan-Pei Lee, \emph{Extremal analytic discs
  with prescribed boundary data}, Trans. Amer. Math. Soc. \textbf{310} (1988),
  no.~1, 355--369. \MR{930081 (89f:32043)}

\bibitem{FS1998}
Siqi Fu and Emil~J. Straube, \emph{Compactness of the
  {$\overline\partial$}-{N}eumann problem on convex domains}, J. Funct. Anal.
  \textbf{159} (1998), no.~2, 629--641. \MR{1659575 (99h:32019)}

\bibitem{GS2013}
H.~{Gaussier} and H.~{Seshadri}, \emph{{On the Gromov hyperbolicity of convex
  domains in $\mathbb{C}^n$}}, ArXiv e-prints (2013).

\bibitem{G1997}
Herv{\'e} Gaussier, \emph{Characterization of convex domains with noncompact
  automorphism group}, Michigan Math. J. \textbf{44} (1997), no.~2, 375--388.
  \MR{1460422 (98j:32040)}

\bibitem{G2015}
William Goldman, \emph{Geometric structures on manifolds},
  \url{http://www.math.umd.edu/~wmg/gstom.pdf}, 2015, Accessed: 2016-01-12.

\bibitem{GK1982}
Robert~E. Greene and Steven~G. Krantz, \emph{Stability of the {C}arath\'eodory
  and {K}obayashi metrics and applications to biholomorphic mappings}, Complex
  analysis of several variables ({M}adison, {W}is., 1982), Proc. Sympos. Pure
  Math., vol.~41, Amer. Math. Soc., Providence, RI, 1984, pp.~77--93.
  \MR{740874 (85k:32043)}

\bibitem{JP2013}
Marek Jarnicki and Peter Pflug, \emph{Invariant distances and metrics in
  complex analysis}, extended ed., de Gruyter Expositions in Mathematics,
  vol.~9, Walter de Gruyter GmbH \& Co. KG, Berlin, 2013. \MR{3114789}

\bibitem{K2001}
Anders Karlsson, \emph{Non-expanding maps and {B}usemann functions}, Ergodic
  Theory Dynam. Systems \textbf{21} (2001), no.~5, 1447--1457. \MR{1855841
  (2002f:37055)}

\bibitem{KN2002}
Anders Karlsson and Guennadi~A. Noskov, \emph{The {H}ilbert metric and {G}romov
  hyperbolicity}, Enseign. Math. (2) \textbf{48} (2002), no.~1-2, 73--89.
  \MR{1923418 (2003f:53061)}

\bibitem{K1977}
Shoshichi Kobayashi, \emph{Intrinsic distances associated with flat affine or
  projective structures}, J. Fac. Sci. Univ. Tokyo Sect. IA Math. \textbf{24}
  (1977), no.~1, 129--135. \MR{0445016 (56 \#3361)}

\bibitem{K2005}
\bysame, \emph{Hyperbolic manifolds and holomorphic mappings}, second ed.,
  World Scientific Publishing Co. Pte. Ltd., Hackensack, NJ, 2005, An
  introduction. \MR{2194466 (2006m:32008)}

\bibitem{L1986}
L{\'a}szl{\'o} Lempert, \emph{Complex geometry in convex domains}, Proceedings
  of the {I}nternational {C}ongress of {M}athematicians, {V}ol. 1, 2
  ({B}erkeley, {C}alif., 1986) (Providence, RI), Amer. Math. Soc., 1987,
  pp.~759--765. \MR{934278 (89f:32045)}

\bibitem{M1993}
Peter~R. Mercer, \emph{Complex geodesics and iterates of holomorphic maps on
  convex domains in {${\bf C}^n$}}, Trans. Amer. Math. Soc. \textbf{338}
  (1993), no.~1, 201--211. \MR{1123457 (93j:32035)}

\bibitem{NTT2014}
N.~{Nikolov}, P.~J. {Thomas}, and M.~{Trybula}, \emph{{Gromov
  (non)hyperbolicity of certain domains in $\mathbb{C}^2$}}, ArXiv e-prints
  (2014).

\bibitem{NPZ2011}
Nikolai Nikolov, Peter Pflug, and W{\l}odzimierz Zwonek, \emph{Estimates for
  invariant metrics on {$\Bbb C$}-convex domains}, Trans. Amer. Math. Soc.
  \textbf{363} (2011), no.~12, 6245--6256. \MR{2833552 (2012m:32009)}

\bibitem{NT2015}
Nikolai Nikolov and Maria Trybu{\l}a, \emph{The {K}obayashi balls of
  ({$\Bbb{C}$}-)convex domains}, Monatsh. Math. \textbf{177} (2015), no.~4,
  627--635. \MR{3371366}

\bibitem{R1971}
H.~L. Royden, \emph{Remarks on the {K}obayashi metric}, Several complex
  variables, {II} ({P}roc. {I}nternat. {C}onf., {U}niv. {M}aryland, {C}ollege
  {P}ark, {M}d., 1970), Springer, Berlin, 1971, pp.~125--137. Lecture Notes in
  Math., Vol. 185. \MR{0304694 (46 \#3826)}

\bibitem{V1989}
Sergio Venturini, \emph{Pseudodistances and pseudometrics on real and complex
  manifolds}, Ann. Mat. Pura Appl. (4) \textbf{154} (1989), 385--402.
  \MR{1043081 (91d:32032)}

\bibitem{Z2014}
Andrew~M. Zimmer, \emph{Gromov hyperbolicity and the {K}obayashi metric on
  convex domains of finite type}, ${\rm to \ appear \ in}$ Math. Ann. (2015),
  1--74 (English).

\end{thebibliography}

\end{document}